\documentclass[11pt,a4paper]{amsart}
\usepackage{texmac,slashbox,enumerate}
\operators{Stab,Prim,Inn,Out,Inf}

\bbsymbols{b}{C,F,Z,Q}
\calsymbols{c}{H,U,X,Y,P,M}
\fraksymbols{f}{p}
\def\sm{\tiny}

\begin{document}

\title[Elliptic curves with $p$-Selmer growth for all $p$]{Elliptic curves with $p$-Selmer growth\\ for all $p$}
\author{Alex Bartel}
\address{Department of Mathematics, POSTECH, Pohang, Gyungbuk 790-784, Republic of Korea}
\email{alxbartel@gmail.com}
\llap{.\hskip 10cm} \vskip -0.8cm
\maketitle

\begin{abstract}
It is known, that for every elliptic curve over $\Q$ there exists a quadratic
extension in which the rank does not go up. For a large class of elliptic
curves, the same is known with the rank replaced by the size of the 2-Selmer
group. We show, however, that there exists a large supply of semistable
elliptic curves $E/\Q$ whose 2-Selmer group grows in size in every bi-quadratic
extension, and such that moreover, for any odd prime $p$, the size of the
$p$-Selmer group grows in every $D_{2p}$-extension and every elementary
abelian $p$-extension of rank at least 2. We provide a simple criterion for
an elliptic curve over an arbitrary number field to exhibit this behaviour.
We also discuss generalisations to other
Galois groups.
\end{abstract}

\section{Introduction}
In \cite{noGoldfeld} it is shown that there exist elliptic curves over number
fields for which in every quadratic extension of the base field
either the rank goes up or the Tate-Shafarevich group becomes infinite.
Equivalently, every quadratic
twist of such a curve has either positive rank or an infinite Tate-Shafarevich
group (the latter is of course conjectured to never happen).
Over $\Q$, such curves do not
exist by the combined work of Kolyvagin \cite{Kolyvagin} and
Bump--Friedberg--Hoffstein--Murty--Murty--Waldspurger \cite{BFH,MurtyMurty,Waldspurger}.
In fact, it is conjectured that half of all quadratic twists of an elliptic
curve over $\Q$ have rank 0. Moreover, if $E/\Q$ has no cyclic 4-isogeny,
then there exists a quadratic extension $F/\Q$ such that the size of the 2-Selmer
group of $E$ over $F$ is the same as over $\Q$ \cite[Theorem 1]{Sw-D},
\cite[Theorem 1.5]{MR}, \cite[Theorems 1.1 and 1.3]{Klagsbrun}.

As we shall show, however, if we allow only slightly bigger extensions
of $\Q$ than quadratic, then there are lots of elliptic curves over $\Q$
whose Selmer groups grow in size in all such extensions.
Below, $S^p(E/F)$ will
denote the $p$-Selmer group of $E$ over a number field $F$, which we recall
to be an elementary abelian $p$-group defined as the the kernel
of the map $H^1(F,E[p])\longrightarrow \prod_{\fp} H^1(F_{\fp},E)$.
Here, the product runs over all places of $F$, and each $\Gal(\bar{F}_\fp/F_\fp)$
is identified with a subgroup of $\Gal(\bar{F}/F)$.

\begin{theorem}\label{thm:main}
Let $E/K$ be a semistable elliptic curve over a number field with positive rank.
Let $p$ and $G$ be one of the following combinations of a prime number and a
finite group:
\begin{enumerate}[(a)]
  \item\label{item:c2c2} $p=2$, $G\cong C_2\times C_2$;
  \item\label{item:d2p} $p$ is odd, $G\cong D_{2p}$, the dihedral group of order $2p$;
  \item\label{item:cpcp} $p$ is odd, and either $G\cong C_p\times C_p$ or $G\cong C_p\rtimes C_q$, where
  $q$ is an odd prime, and $C_q$ acts faithfully on $C_p$.
\end{enumerate}
Suppose further that if $p$ and $G$ are as in (\ref{item:d2p}), then the rank
of $E/K$ is greater than the number of primes of $K$ at which $E$ has non-split
multiplicative reduction, while if $p$ and $G$ are as in (\ref{item:c2c2}), then
the rank of $E/K$ is greater than the number of primes $v$ of non-split
multiplicative reduction for which $\ord_v(\Delta(E))$ is even.
Then, the following conclusions hold.
\begin{enumerate}
\item\label{item:shachange} If $\sha(E/K)[p^\infty]$ is finite, then we have that for every Galois
extension $F/K$ with Galois group $G$, either the order of the $p$-primary
part of the Tate-Shafarevich group changes at some step in $F/K$, or
$E(F)\supsetneq E(K)$.
\item\label{item:shagrowth} If $\sha(E/K)[p]=0$, then for every
Galois extension $F/K$ with Galois group $G$, either $\#\sha(E/F)[p^\infty] >0$,
or $E(F)\supsetneq E(K)$.
\item\label{item:selmergrowth} If $\sha(E/K)[p]=0$ and $E(K)[p]=0$, then
for every Galois extension $F/K$ with Galois group $G$, we have $\#S^p(E/F) > \#S^p(E/K)$.
\end{enumerate}
\end{theorem}

\begin{example}
The first few curves over $\Q$ in Cremona's database that satisfy the
hypotheses of Theorem \ref{thm:main} (\ref{item:shachange}) and (\ref{item:shagrowth})
for all the combinations of $p$ and $G$ listed there
are 91b1, 91b2, 91b3, 123a1, 123a2, 141a1, 142a1,
155a1, all of rank 1 and with no primes of non-split multiplicative reduction,
and with trivial Tate-Shafarevich groups \cite{bsd}. Out of these, 91b3,
123a2, 141a1, 142a1 have trivial torsion subgroup over $\Q$ and thus also satisfy
the stronger hypothesis of Theorem \ref{thm:main} (\ref{item:selmergrowth})
for all the pairs $p$, $G$.

Also, the huge majority of rank 2 curves over $\Q$ are expected to have trivial
Tate-Shafarevich groups \cite{Del07}. For example the curve with Cremona label
389a1 almost certainly satisfies all the hypotheses of Theorem (\ref{thm:main})
for all $p$, $G$. Unfortunately, we do not even know
that the Tate-Shafarevich group is finite for a single rank 2 elliptic curve.
In principle, triviality of the $p$-part for any fixed prime $p$ can
be checked algorithmically using descent (see \cite{StollSchaefer} and the references therein).
Such checks have been performed for thousands of higher rank curves, including 389a1,
and thousands
of primes in \cite{SteinWuthrich} (using methods quite different from descent).
\end{example}

\begin{example}
  The following example illustrates that the above results are not a parity
  phenomenon. Let $E$ be the curve with Cremona label 65a1. It has rank 1
  over $\Q$ and square-free discriminant, and $\sha(E/\Q)$ is trivial, so the
  conclusion of Theorem \ref{thm:main} (\ref{item:shagrowth}) holds 
  for all combinations of $p$ and $G$ listed in the theorem, and in particular for
  $p=2$, $G\cong C_2\times C_2$. Let $F$ be the bi-quadratic field
  $\Q(\sqrt{3},\sqrt{5})$. Then we have $E(\Q) = E(F)$. However,
  $\sha(E/F)[2]\cong C_2\times C_2$, as predicted by Theorem \ref{thm:main}.
  This cannot be detected by root numbers.
\end{example}

We will collect the necessary ingredients of the proof in great generality,
with no assumption on the Galois group of the extension $F/K$, although we 
will simplify the exposition by assuming early on that $E$ is semistable.
Then we will perform the necessary calculations in the case of dihedral
and bi-quadratic extensions, thereby proving Theorem \ref{thm:main} for
(\ref{item:c2c2}) and (\ref{item:d2p}).

In the last section, we will discuss generalisations to other Galois groups,
such as those of Theorem \ref{thm:main} (\ref{item:cpcp}).
We will also explain why our approach cannot be
pushed any further than that. This will rely on a certain representation
theoretic classification \cite[Corollary 9.2]{BraRels}.
\begin{acknowledgements}
  I would like to thank Vladimir Dokchitser for several very helpful remarks.
  I would also like to thank the anonymous referee for carefully reading the
  manuscript and for suggestions that helped   greatly improve the exposition.
\end{acknowledgements}

\section{Tamagawa numbers and regulators}

We begin with a simplifying reduction to a seemingly special case:
\begin{lemma}\label{lemma:simpler}
  Let $E/K$ be a semistable elliptic curve over a number field.
  Let the prime $p$ and the finite group $G$ be one of the combinations of
  Theorem \ref{thm:main}. Theorem \ref{thm:main} holds for all extensions
  $F/K$ with Galois group $G$ if and only if it holds for all $F/K$ with
  Galois group $G$ that satisfy the following additional conditions:
  \begin{enumerate}
    \item\label{item:sha} $\sha(E/F)[p^\infty]$ is finite (and consequently, by the
    inflation-restriction exact sequence, so are the $p$-primary parts of the
    Tate-Shafarevich groups over all subfields);
    \item $E(K)\otimes \Z_p = E(F)\otimes \Z_p$ under the natural inclusion map.
  \end{enumerate}
\end{lemma}
\begin{proof}
  If $\sha(E/K)[p^\infty]$ is infinite, then Theorem \ref{thm:main} is empty.
  Also, if $\sha(E/K)[p^\infty]$ is finite, then clearly the conclusions of
  Theorem \ref{thm:main} automatically hold for all $F/K$ with Galois group
  $G$ for which $\sha(E/F)[p^\infty]$ is infinite. This proves (\ref{item:sha}).
  
  Finally, if $E(K)\otimes \Z_p \neq E(F)\otimes \Z_p$, then either
  $\rk(E/F)>\rk(E/K)$, or a point of infinite order on $E(K)$ becomes
  $p$-divisible over $F$, or $\#E(F)[p^\infty] > \#E(K)[p^\infty]$. In all
  three cases, the conclusions (\ref{item:shachange}) and (\ref{item:shagrowth})
  of Theorem \ref{thm:main} obviously hold.
  As for (\ref{item:selmergrowth}), certainly if
  $\rk(E/F)>\rk(E/K)$ or if $\#E(F)[p^\infty] > \#E(K)[p^\infty]$, then
  the assumption $\sha(E/K)[p]=0$ forces $\#S^p(E/F) > \#S^p(E/K)$. If on the
  other hand a point of infinite order on $E(K)$ becomes $p$-divisible, then
  by Kummer theory, $E(F)[p]$ must be non-trivial, so the assumption
  $E(K)[p]=0$ again forces $\#S^p(E/F) > \#S^p(E/K)$.
\end{proof}

We will therefore henceforth restrict our attention to semistable elliptic
curves $E/K$ and to Galois extensions $F/K$ such that $E/F$ satisfies the
additional conditions of Lemma \ref{lemma:simpler}.

\begin{definition}
Let $G$ be a finite group. A formal $\Z$-linear combination of representatives
of conjugacy classes of
subgroups $\Theta=\sum_H n_HH$ is called a \emph{Brauer relation} if the virtual
permutation representation $\bigoplus \C[G/H]^{\oplus n_H}$ is zero.
\end{definition}
For a detailed discussion of the concept of Brauer relations, see the
introduction to \cite{BraRels}.

Let $E/K$ be an elliptic curve over a number field, $F/K$ a Galois extension
with Galois group $G$, and $\Theta=\sum_H n_HH$ a Brauer relation.
There is a corresponding relation between $L$-functions of $E$ over the
intermediate fields:
\[
\prod_H L(E/F^H,s)^{n_H}=1.
\]
If $E$ is semistable and has finite Tate-Shafarevich groups over all
intermediate extensions of $F/K$, then a combination of various compatibility
results on the Birch
and Swinnerton-Dyer conjecture yields a relation\footnote{see \cite[Theorem 2.3]{squarity} and the
remarks at the beginning of \cite[\S 2.2]{squarity}, as well as \cite[Remark 2.3]{Bar-09}}
between arithmetic invariants of $E$
over the intermediate fields. We shall recall the necessary notation shortly:
\begin{eqnarray}\label{eq:bsd}
\prod_H \left(\frac{c(E/F^H)\#\sha(E/F^H)\Reg(E/F^H)}{|E(F^H)_{\tors}|^2}\right)^{n_H} = 1.
\end{eqnarray}
Moreover, if only the $p$-primary parts of the Tate-Shafarevich groups are
assumed to be finite for some prime $p$, then the $p$-part of equation
(\ref{eq:bsd}) holds.

Here, $c$ denotes the product of Tamagawa numbers\footnote{It is here that
we use the assumption that $E$ is semistable. Otherwise, the Tamagawa numbers
have to be re-normalised. See \cite[Theorem 2.3]{squarity} for the general case.} of $E$
over the finite places of the respective field.

Recall that the regulator of an elliptic curve is the determinant of the
N\'eron-Tate height pairing evaluated on any basis of the free part of
the Mordell-Weil group. Note that since each of the regulators is some real
number, in general transcendental, it does not make any sense to speak of
its $p$-part. However, since the quotient $\prod_H \Reg(E/F^H)^{n_H}$ is a rational
number (this is an immediate consequence of \cite[Theorem 2.17]{tamroot}),
it does make sense to speak of the $p$-parts of the regulator quotient and
of the remaining terms.

The precise normalisation of the N\'eron-Tate height
pairing that enters the Birch and Swinnerton-Dyer conjecture
will be crucial for us. If $M/K$ is a finite extension of fields,
and if $\langle\cdot,\cdot\rangle_K$, respectively $\langle\cdot,\cdot\rangle_M$
denotes the N\'eron-Tate height pairing on $E(K)$, respectively on $E(M)$,
then for any $P,Q\in E(K)$, $\langle P,Q\rangle_M = [M:K]\langle P,Q\rangle_K$.
In particular, if $E/K$ does not acquire any new points of infinite order over $F$,
then the regulator quotient in (\ref{eq:bsd}) does not vanish in general, but rather
\begin{eqnarray}\label{eq:regconst}
\prod_H \Reg(E/F^H)^{n_H} = \prod_H \frac{1}{|H|^{n_H\rk E(K)}}.
\end{eqnarray}
More generally, if $E(K)\otimes\Z_p = E(F)\otimes\Z_p$,
then the $p$-part of equation (\ref{eq:regconst}) holds.

In summary, if $E/K$ is a semistable elliptic curve with $\sha(E/F)[p^\infty]$
finite and if $E(K)\otimes\Z_p= E(F)\otimes\Z_p$, then
\begin{eqnarray}\label{eq:simpleBSD}
\prod_H \#\sha(E/F^H)^{n_H}c(E/F^H)^{n_H} =_{p'} \prod |H|^{n_H\rk E(K)}.
\end{eqnarray}
Here, $=_{p'}$ means that both sides have the same $p$-adic valuation.

\section{Dihedral and bi-quadratic extensions}
\subsection{Dihedral extensions}
Suppose that $G\cong D_{2p}$, where $p$ is an odd prime. There is a Brauer relation
in $G$ of the form $\Theta = 1 - 2C_2 - C_p + 2G$. For this relation, we have
\[
\prod |H|^{n_H} = \frac{(2p)^2}{4p} = p,
\]
so that the right hand side of equation (\ref{eq:simpleBSD}) is $p^{\rk E(K)}$.
If $v$ is a place of $K$, write $c_v(E/K)$ for the Tamagawa
number at $v$, and $c_v(E/F^H)$ for the product of Tamagawa numbers at
places of $F^H$ above $v$. Write $c_v(E/\Theta)$ for
the quotient $\frac{c_v(E/K)^2c_v(E/F)}{c_v(E/F^{C_2})^2c_v(E/F^{C_p})}$.
Similarly, write
$\#\sha(E/\Theta)[p^\infty]$ for the $p$-primary part of the corresponding
quotient of sizes of Tate-Shafarevich groups.
Finally, let $M$ denote the intermediate quadratic extension $M=F^{C_p}$.
The following table gives the possible values of the $p$-part of $c_v(E/\Theta)$,
depending on the reduction type of $E$ at $v$ (horizontal axis) and on the
splitting behaviour of $v$ in $F/K$ (vertical axis):\vspace{6mm}\newline
\begin{tabular}{l | c | c | c }
 \backslashbox[0pt][l]{\tiny{splitting of $v$}\kern-2em}{\kern-1.5em\tiny{redn. type of $E$}} & 
\begin{tabular}{c}\sm{split mult.}\\ \sm{over $K$} \end{tabular} & \begin{tabular}{c}
\sm{non-split mult.}\\ \sm{over $F$} \end{tabular}
 & \begin{tabular}{c}\sm{non-split mult. over $K$,}\\
  \sm{split mult. over $M$}\end{tabular}\\
 \hline 
 \sm{splits into more than one prime} & 1 & 1 & 1 \\
 \hline
\sm{inert in $M/K$, ramified in $F/M$}
& $1/p$ & --- & $p$ \\
 \hline
 \sm{totally ramified in $F/K$} & $1/p$ & 1 & --- 
\end{tabular}
\vspace{6mm}\newline
It follows immediately from this table and from equation (\ref{eq:simpleBSD})
that if $E/K$ is semistable, if $\sha(E/F)[p^\infty]$ is finite, if
$E(K)\otimes \Z_p= E(F)\otimes \Z_p$, and if the rank of $E/K$ is greater
than the number of primes of non-split multiplicative reduction,
then $\#\sha(E/\Theta)[p^\infty]>1$, and thus
at least one of $\#\sha(E/K)[p^\infty]$, $\#\sha(E/F)[p^\infty]$ is strictly
larger than at least one of $\#\sha(E/M)[p^\infty]$, $\#\sha(E/F^{C_2})[p^\infty]$.
This, together with Lemma \ref{lemma:simpler} concludes the proof of
Theorem \ref{thm:main} for $G\cong D_{2p}$.

\subsection{Bi-quadratic extensions}
We now apply the same reasoning to $G\cong C_2\times C_2$, with $p=2$.
Denote the three distinct subgroups of order 2 by
$C_2^a$, $C_2^b$, $C_2^c$. The space of Brauer relations in $G$ is generated
by the relation $\Theta = 1 - C_2^a - C_2^b - C_2^c +2G$, for which we have
\[
\prod |H|^{n_H} = \frac{16}{8} = 2.
\]
Again writing $c_v(E/\Theta)$ for the corresponding quotient of Tamagawa numbers
of $E$ at places above $v$ over the corresponding intermediate fields of $F/K$,
the following table gives the possible values of $c_v(E/\Theta)$:
\vspace{6mm}\newline
\begin{tabular}{l | c | c | c }
 \backslashbox[0pt][l]{\tiny{splitting of $v$}\kern-2em}{\kern-1.5em\tiny{redn. type of $E$}} & 
\begin{tabular}{c}\sm{split mult.}\\ \sm{over $K$} \end{tabular} & \begin{tabular}{c}
\sm{non-split mult.}\\ \sm{over $F$} \end{tabular}
 & \begin{tabular}{c}\sm{non-split mult. over $K$,}\\
  \sm{split mult. over some $F^{C_2}$}\end{tabular}\\
 \hline 
 \sm{splits in some $F^{C_2}$} & 1 & 1 & 1 \\
 \hline
\begin{tabular}{l}
\hskip -0.45em\sm{inert in some $F^{C_2}/K$,}\\ \hskip -0.45em\sm{ramified in $F/F^{C_2}$}
\end{tabular}
& $1/2$ & --- & \begin{tabular}{@{}l@{}l}
  \sm{2:} &\sm{ $\ord_v(\Delta(E))$ is even}\\
  \sm{1/2:} & \sm{ otherwise} \end{tabular} \\
 \hline
 \sm{totally ramified in $F/K$} & $1/2$ &
  \begin{tabular}{@{}l@{}l@{}}
  \sm{1:} &\sm{ $\ord_v\Delta(E)$ even}\\
  \sm{1/4:} & \sm{ otherwise} \end{tabular}
  & --- 
\end{tabular}\vspace{6mm}\newline
As above, this table together with equation (\ref{eq:simpleBSD}) and with
Lemma \ref{lemma:simpler} proves Theorem \ref{thm:main} for $G\cong C_2\times C_2$.

\section{Generalisation to other Galois groups}
If $G$ is a subgroup of a group $\tilde{G}$, then by transitivity of
induction, a Brauer relation $\Theta$ in $G$ automatically gives a Brauer
relation $\Ind_{\tilde{G}/G}\Theta$ in $\tilde{G}$. Also, if $G$ is a quotient
of a group $\Gamma$, $G=\Gamma/N$, then a Brauer relation $\Theta=\sum_H n_HH$
in $G$ gives rise to a Brauer relation $\Inf_{\Gamma/N}\Theta=\sum_Hn_HNH$ in $\Gamma$.

In general, in order to prove by the same technique as above that the size
of the $p$-Selmer group
of some elliptic curve grows in all Galois extensions with Galois group $G$,
we need to have a Brauer relation $\Theta=\sum_H n_H H$ in $G$ such that
$\ord_p(\prod |H|^{n_H})\neq 0$. This quantity is clearly invariant under
inductions and lifts of Brauer relations.

\begin{proposition}[\cite{BraRels}, Corollary 9.2]\label{prop:brarels}
  Let $p$ be a prime number. A finite group $\tilde{G}$ has a Brauer relation
  $\Theta=\sum_H n_H H$ with $\ord_p(\cC_\Theta(\triv))\neq 0$ if and only
  if $\tilde{G}$ has a subquotient $G$ isomorphic either to $C_p\times C_p$ or
  to $C_p\rtimes C_q$ with $C_q$ cyclic of prime order
  acting faithfully on $C_p$.
  Moreover, in the former case $\Theta$ can be taken to be induced and/or
  lifted from the relation
  $1 - \sum_{U\leq_p G} U + pG$; 
  while in the latter case $\Theta$ can be taken to be induced and/or lifted
  from the relation
  $1 - q C_q - C_p + qG$.
\end{proposition}

Suppose that $\tilde{G}$ is a finite group that has a quotient isomorphic to
$G$. If some inequalities between the rank of an elliptic curve and the number
of places of certain reduction types ensure that sizes of Tate-Shafarevich 
groups or Mordell-Weil groups grow in all $G$-extensions of number fields,
as in Theorem \ref{thm:main}, then clearly the same inequalities imply the
same conclusions for $\tilde{G}$-extensions.

On the other hand, suppose that $G$ is a subgroup of a finite group $\tilde{G}$.
Suppose that some condition on the rank of $E$ and on the
number of places of given reduction types implies the conclusions of
Theorem \ref{thm:main} (\ref{item:shachange}), say, for all Galois extensions
with Galois group $G$.
Some care is then needed to deduce the same result for $\tilde{G}$-extensions.
Indeed, if $F/K$ is Galois with Galois group $\tilde{G}$, then one would like
to deduce the conclusion of Theorem \ref{thm:main} (\ref{item:shachange}) for $F/K$ by
applying the theorem to the extension $F/F^G$. But in order to satisfy the required
inequalities for the number of places of $F^G$, one might need to impose stronger
conditions on $E/K$, since the number of places of given reduction type might grow
in $F^G/K$ (by at most a factor of $[\tilde{G}:G]$). This is of course a
straightforward modification.

Thus, we may restrict
attention to the groups $C_p\times C_p$ and $C_p\rtimes C_q$ as in
Proposition \ref{prop:brarels}.

When $G\cong C_p\times C_p$ and $\Theta=1 - \sum_{U\leq_p G} U + pG$, we have
\[
\prod |H|^{n_H} = \frac{(p)^{2p}}{p^{p+1}} = p^{p-1}.
\]
When $G\cong C_p\rtimes C_q$ and $\Theta=1 - q C_q - C_p + qG$, we have
\[
\prod |H|^{n_H} = \frac{(pq)^q}{pq^q} = p^{q-1}.
\]

Having already dealt with such groups of even order,
we may now restrict our attention to groups of odd order, so that only the
primes of $K$ at which $E$ has split multiplicative reduction contribute to
the $p$-part of the corresponding Tamagawa number quotients.

Here are the possible values of $c_v(E/\Theta)$ when $E/K$ has split
multiplicative reduction at $v$, first for $G\cong C_p\times C_p$
and $\Theta=1 - \sum_{U\leq_p G} U + pG$, and then for $G\cong C_p\rtimes C_q$
and $\Theta=1 - q C_q - C_p + qG$, where $p$ and $q$ are odd primes:\vspace{6mm}\newline
\begin{tabular}{c | c | c}
\sm{$v$ splits} & \sm{$v$ is inert in some $F^{C_p}/K$ and ramified in $F/F^{C_p}$}
& \sm{$v$ is totally ramified}\\
\hline
1 & $p^{1-p}$ & $p^{1-p}$
\end{tabular},\vspace{6mm}\newline
\begin{tabular}{c | c | c}
\sm{$v$ splits} & \sm{$v$ is inert in $F^{C_p}/K$ and ramified in $F/F^{C_p}$}
& \sm{$v$ is totally ramified}\\
\hline
1 & $p^{1-q}$ & $p^{1-q}$
\end{tabular}.\vspace{6mm}\newline
These tables together with equation (\ref{eq:simpleBSD}) finish the proof of
Theorem \ref{thm:main} for $G\cong C_p\times C_p$ and $G\cong C_p\rtimes C_q$.

\end{document}